\documentclass[reqno]{amsart}

\usepackage{amssymb,amsfonts, amsmath, amsthm}
\usepackage[all,arc]{xy}
\usepackage{enumerate}
\usepackage{graphicx, enumerate}
\usepackage[utf8]{inputenc}
\usepackage[english]{babel}
\usepackage{graphicx}
\usepackage{mathtools}
\usepackage{etoolbox}
\usepackage{bm}
\usepackage{tikz}

\newtheorem{thm}{Theorem}[section]
\newtheorem*{thm*}{Theorem}

\newtheorem{cor}[thm]{Corollary}
\newtheorem*{cor*}{Corollary}
\newtheorem{prop}[thm]{Proposition}
\newtheorem*{prop*}{Proposition}
\newtheorem{lem}[thm]{Lemma}
\newtheorem*{lem*}{Lemma}

\newtheorem*{conj*}{Conjecture}

\newtheorem*{quest*}{Question}

\theoremstyle{definition}

\theoremstyle{remark}
\newtheorem{rem}[thm]{Remark}

\newtheorem*{ack*}{Acknowledgements}

\newcommand{\Z}{\mathbb{Z}}

\newcommand{\F}{\mathcal{F}} 
\newcommand{\G}{\mathcal{G}}
\newcommand{\X}{\mathbb{X}} 
\newcommand{\ZZ}{\mathcal{Z}} 
\newcommand{\T}{\mathcal{T}} 

\title[Betti tables for matrix factorizations of $XY(X-Y)(X-\lambda Y)$]{Betti tables for indecomposable matrix factorizations of $XY(X-Y)(X-\lambda Y)$}
\author{Vincent G\'elinas}
\thanks{The author is affiliated with the University of Toronto.}
\date{}

\begin{document}

\begin{abstract} 
	We classify the Betti tables of indecomposable graded matrix factorizations over the simple elliptic singularity $f_\lambda = XY(X-Y)(X-\lambda Y)$ by making use of an associated weighted projective line of genus one.  
\end{abstract}

\maketitle
\tableofcontents

\section{Introduction}
Let $\Lambda$ be a Cohen-Macaulay algebra of infinite CM-type, occuring as the local algebra of a reduced curve singularity over an algebraically closed field $k$. Drozd and Greuel \cite{DG} proved (in $char\ k \neq 2$) that $\Lambda$ is of tame CM-type if and only if it birationally dominates a curve singularity of type $T_{pq}$ given by the equation $X^p + \alpha X^2Y^2 + Y^q = 0\ (\alpha \neq 0, 1)$. Furthermore, $\Lambda$ is tame of finite growth if and only if it dominates a curve of type $T_{44}$ or $T_{36}$. Tameness of the $T_{pq}$ singularities can be shown indirectly via deformation theory, with abstract classifications known from \cite{DGav}. Tameness of $T_{44}$ in particular was established by Dieterich \cite{Diet} who related it to a particular tubular quiver path algebra. In \cite{DT}, Drozd and Tovpyha raised the question of finding explicit presentations of indecomposable maximal Cohen-Macaulay modules over the completion of $T_{44} \sim XY(X-Y)(X-\lambda Y)\ (\lambda \neq 0, 1)$, or equivalently to produce the indecomposable matrix factorizations of $f_\lambda$. They reduced this problem to a ``matrix problem'', namely representations of bunches of chains, and used this to produce some of the indecomposables.

In this paper we investigate the indecomposable graded matrix factorizations of $T_{44}$ using triangulated category methods, and produce a classification closer in spirit to the work of Dieterich. One can deal with $T_{36}$ with similar methods but for simplicity we will restrict ourselves to the case $T_{44}$. We deduce a complete classification of Betti tables of indecomposable graded MCM modules over $k[X,Y]/f_{\lambda}$. We will see that all MCM modules over the algebra $k[[X,Y]]/f_\lambda$ are gradable and so this subsumes the above problem. The main result of this paper is (see \ref{bettifirstkind}, \ref{bettisecondkind}):
\begin{thm*}There is a classification of Betti tables of indecomposable graded matrix factorizations of $f_\lambda$. Up to internal degree shifts, these fall into $5$ general and $4$ exceptional shapes (respectively $3$ and $2$ up to shifts and syzygies). 
\end{thm*}
We also describe the set of indecomposables with fixed Betti table, characterize exceptional objects by a numerical criterion, study natural functionals on $K_0$ and give a classification of the indecomposable Ulrich modules over $k[X,Y]/f_\lambda$. Our work is close in nature to that of A. Pavlov who classified the Betti tables of matrix factorizations of a Hesse plane cubic in \cite{Pav}, although the presence of exceptional objects in our setting adds a layer of complexity.  

Order the linear factors as $XY(X-Y)(X-\lambda Y) = l_1l_2l_3l_4$. We will study the classification of indecomposable graded matrix factorizations of $f_\lambda$ by means of the finite dimensional ``Squid algebra'' $Sq(2,2,2,2; \lambda)$, which is the path algebra of the quiver
	
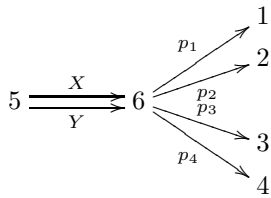
\[
		\xymatrix@C3pc@R0.3pc{	&& 1 \\ 
				&& 2 \\
				5 \ar@<.5ex>[r]^{X} \ar@<-.5ex>[r]_{Y} & 6 \ar[ruu]^{p_1} \ar[ru]_{p_2} \ar[rd]^{p_3} \ar[rdd]_{p_4} \\
				&& 3 \\
				&& 4 \\
}
	\]
	subject to the relations $p_i l_i(X,Y) = 0$. The algebra $Sq(2,2,2,2;\lambda)$ arises as the endomorphism algebra of a tilting sheaf on the weighted projective line $\X = \mathbb{P}^1(2,2,2,2;\lambda)$ introduced by Geigle and Lenzing, and this will reduce most of our work to sheaf cohomology calculations on the latter. More precisely we will use

	\begin{thm*}[Buchweitz-Iyama-Yamaura] There is a tilting MCM module $T$ with $\underline{{\rm End}}_{gr R}(T) \cong Sq(2,2,2,2;\lambda)$, and so we have induced equivalences of triangulated categories
		\[
			\underline{{\rm MCM}}(gr\ S/f_\lambda) \cong {\rm D}^b(Sq(2,2,2,2;\lambda)) \cong {\rm D}^b(\X).
		\]
	\end{thm*}
This result is a special case of a recent result of Buchweitz-Iyama-Yamaura \cite{BIY}, and the author has learned it from Buchweitz. That matrix factorizations of $f_\lambda$ be related to $\X = \mathbb{P}^1(2,2,2,2;\lambda)$ should also be well-known to experts. Indeed one can, at least in $char\ k \neq 2$, compare $\Z$-graded matrix factorizations of $f_\lambda$ to $\Z \oplus \Z_2$-graded matrix factorizations of $f_\lambda + z^2$ by results of Kn\"orrer. The elliptic curve $E_\lambda = \{ f_\lambda + z^2 = 0 \}$ in $\mathbb{P}(1,1,2)$ is a branched double cover $\pi: E_\lambda \to \mathbb{P}^1$ ramified over $\{f_\lambda = 0 \}$, with hyperelliptic involution $\sigma$. One should then obtain from Kn\"orrer's and Orlov's theorem an equivalence 
	\[
		{\rm HMF}^{\Z}(S, f_\lambda) \cong {\rm HMF}^{\Z \oplus \Z_2}(S[z], f_\lambda + z^2) \cong {\rm D}^b({\rm coh}_\sigma (E_\lambda))
\]
with the derived category of $\sigma$-equivariant coherent sheaves on $E_\lambda$. Coherent sheaves on the weighted projective line $\mathbb{P}^1(2,2,2,2;\lambda)$ are equivalent to the latter in $char\ k \neq 2$, and so we will prefer the above characteristic-free setup via tilting theory, where calculations are more tractable in any case. 


\section{Background and setup}
Throughout this paper, modules will refer to finitely generated right graded modules. We denote by $M \mapsto M(i)$ the grade shift autoequivalence with $M(i)_n = M_{n+i}$. Fix a field $k$ and let $f \in S = k[x_1, \dots, x_n]$ be a homogeneous polynomial of degree $d$. The triangulated homotopy category of graded matrix factorizations ${\rm HMF}^{\Z}(S, f)$ has for objects pairs of morphisms $(\varphi, \psi)$ of $S$-free modules 
\[
	F \xleftarrow{\varphi} G \xleftarrow{\phi} F(-d)
\]
with compositions $\varphi \psi = f \cdot {\rm Id}$ and $\psi \varphi = f \cdot {\rm Id}$, and for morphisms the natural notion of homotopy classes of chain-maps. Matrix factorizations give $S$-free presentations of graded maximal Cohen-Macaulay (MCM) modules over $R = S/f$, with the short resolution
\[
	0 \leftarrow M \leftarrow F \xleftarrow{\varphi} G \leftarrow 0
\]
of $M = {\rm coker}(\varphi)$ descending to a $2$-periodic $R$-free resolution 
\[
P_*:	0 \leftarrow M \leftarrow \overline{F} \xleftarrow{\varphi} \overline{G} \xleftarrow{\psi} \overline{F}(-d) \xleftarrow{\varphi} \overline{G}(-d) \leftarrow \dots
\]
This induces an equivalence of triangulated categories ${\rm HMF}^{\Z}(S, f) \cong \underline{{\rm MCM}}(gr R)$ with the projectively stable category of graded MCM modules. The latter inherits a triangulated structure from Buchweitz's equivalence \cite{buchweitzmanuscript}
\[
	{\rm D}_{sg}(gr R) = {\rm D}^b(gr R)/{\rm D}^{perf}(gr R) \cong \underline{{\rm MCM}}(gr R)
\]
or equivalently from an equivalence with the homotopy category of complete resolutions, obtained by extending the above resolution of $M$ by periodicity to the acyclic complex of graded free modules
\[
	C_*: 	\cdots \leftarrow F(d) \xleftarrow{\varphi} G(d) \xleftarrow{\psi} \overline{F} \xleftarrow{\varphi} \overline{G} \xleftarrow{\psi} \overline{F}(-d) \xleftarrow{\varphi} \overline{G}(-d) \leftarrow \dots
\]
We have for suspension $\Sigma M = M[1] = {\rm cosyz}_R(M)$ with inverse $\Omega M = M[-1] = {\rm syz}_R(M)$. The category $\underline{{\rm MCM}}(gr R)$ also has the grade shift exact autoequivalence $M \mapsto M(1)$. In particular note the $2$-periodicity natural isomorphism $(d) = [2]$. 

Define the Tate cohomology groups $\underline{{\rm Ext}}_{gr R}^n(M, N)$ for $M$ an MCM module and $N$ any module by ${\rm H}^n{\rm Hom}_{gr R}(C_*, N)$ for any $n \in \Z$. The module $N$ admits an MCM approximation $N^{st}$ fitting in a short exact sequence
\[
	0 \to Q \to N^{st} \to N \to 0
\]
with $Q$ perfect, so that $N$ is sent to $N^{st}$ under Buchweitz's equivalence. Tate cohomology vanishes against perfect complexes, and we have $\underline{{\rm Ext}}^n_{gr R}(M, N) = \underline{{\rm Ext}}^n_{gr R}(M, N^{st}) = \underline{{\rm Hom}}_{gr R}(M, N^{st}[n])$. We define the (complete) graded Betti numbers $\beta_{i,j}$ of $M$ by
\[
	\beta_{i,j} = {\rm dim}_k\ \underline{{\rm Ext}}^i_{gr R}(M, k(-j)) = {\rm dim}_k\ \underline{{\rm Ext}}^i_{gr R}(M, k^{st}(-j)).
\]
When the resolution $P_*$ is minimal, it follows that $\overline{F} = \bigoplus_{j \in \Z}R(-j)^{\oplus \beta_{0, j}}$ and $\overline{G} = \bigoplus_{j \in \Z} R(-j)^{\oplus \beta_{1,j}}$, with the other Betti numbers obtained by periodicity $\beta_{i, j} = \beta_{i+2, j+4}$. Hence calculating Betti tables reduces to calculating dimensions of morphism spaces into the various MCM modules $k^{st}(-j)$. We will need the following:

\begin{prop}\label{Se} Assume that $R$ has isolated singularities. Then $$\mathbb{S}_R(-) = - \otimes_R \omega_R[{\rm dim}\ R - 1] = (a)[{\rm dim}\ R - 1]$$ is a Serre functor for $\underline{{\rm MCM}}(gr R)$, where $a = d-n$.
\end{prop}
Since $R$ is graded connected, the category $\underline{{\rm MCM}}(gr R)$ is Ext-finite and Krull-Schmidt, and in particular idempotent complete. 

\subsection*{Weighted projective lines} 
We refer to \cite{GL, LM, Melt} for basic definitions and results, see also \cite{CB}. In particular we will only use the weighted projective line of genus one $\X = \mathbb{P}^1({\bf p}, \pmb{\lambda}) = \mathbb{P}^1(2,2,2,2; \lambda)$ with $\pmb{\lambda} = (0, 1, \infty, \lambda)$. The derived categories of weighted projective lines of genus one were thoroughly investigated by Lenzing and Meltzer in \cite{LM, Melt}.

We fix the notation for the rest of the paper. Recall that to a general set of weights ${\bf p} = (p_1, \dots, p_n)$ one associates a rank one abelian group $\mathbb{L} = \mathbb{L}({\bf p})$ with 
\[
	\mathbb{L}({\bf p}) = \langle \vec{x}_1, \dots, \vec{x}_n, \vec{c} \ | \ p_1 \vec{x}_1 = \dots = p_n \vec{x}_n = \vec{c} \rangle.
\]
Setting $p = {\rm lcm}(p_1, \dots, p_n)$, there is a group homomorphism $\delta: \mathbb{L} \to \Z$ given by $\delta(\vec{x}_i) = \frac{p}{p_i}$, with finite kernel. We denote by $\vec{\omega} = (n-2)\vec{c} - \sum_{i=1}^n \vec{x}_i$ the canonical element in $\mathbb{L}$. We have an isomorphism of abelian groups $\mathbb{L} \cong {\rm Pic}(\X)$ sending $\vec{x} \mapsto \mathcal{O}(\vec{x})$, and we denote by $\omega_\X = \mathcal{O}(\vec{\omega})$ the canonical line bundle of $\X$. 

Every coherent sheaf on $\X$ is the direct sum of its torsion subsheaf and quotient torsion-free sheaf, which is then a vector bundle. Vector bundles admit finite filtrations by line bundles. The indecomposable torsion sheaves are supported over a single point $x \in \mathbb{P}^1(k)$. We say that $x$ is ordinary if it lies outside of the set $\pmb{\lambda}$, and exceptional otherwise. Torsion sheaves supported over $x$ form a serial abelian subcategory, with unique simple sheaf over $x$ ordinary and $p_i$-many simple sheaves $\{S_{i, j} \}_{j \in \Z/p_i\Z}$ over $x = \lambda_i$ exceptional. These have presentations
\[
	0 \to \mathcal{O}((j-1)\vec{x}_i) \to \mathcal{O}(j\vec{x}_i) \to S_{i, j} \to 0.
\]
In particular we single out $S_{i, 0}$ as the unique simple sheaf with a non-zero section\footnote{This agrees with the notation in \cite{CB} but disagrees with \cite{Melt}.}, and we have ${\rm Hom}(\mathcal{O}, S_{i, 0}) = k$ and $S_{i, j} \otimes \omega_\X = S_{i, j+1}$. 
There is a family of indecomposable ``ordinary '' torsion sheaves $S_x$ for any $x$, with presentation 
\[
	0 \to \mathcal{O}(-\vec{c}) \to \mathcal{O} \to S_x \to 0.
\]
The sheaf $S_x$ has length one when $x$ is ordinary and length $p_i$ over $x = \lambda_i$, with the $S_{i, j}$ as simple composition factors. Lastly, there are additive rank and degree functionals on $K_0(\X) := K_0({\rm coh}(\X))$, uniquely determined by their value on line bundles as
\begin{align*}
	& rk(\mathcal{O}(\vec{x})) = 1 \\
	& deg(\mathcal{O}(\vec{x})) = \delta(\vec{x}). 
\end{align*}
In particular $deg(S_x) = deg(\mathcal{O}(\vec{c})) = p$ and $deg(S_{i,j}) = 1$.

\subsection*{Mutations of exceptional pairs} Given objects $A, B, C$ in an Ext-finite $k$-linear triangulated category $\T$, denote by ${\rm RHom}(A, B) = \bigoplus_{n \in \Z} {\rm Hom}_\T(A, B[n])[-n]$ the object in ${\rm D}^b(k)$, and define
\[
	{\rm RHom}(A, B)\otimes_k C = \bigoplus_{n \in \Z}{\rm Hom}_\T(A, B[n])\otimes_k C[-n].
\]
Interpret the dual ${\rm RHom}(A, B)^*$ in ${\rm D}^b(k)$ accordingly. Following Gorodentsev, we have canonical distinguished triangles in $\T$
\[
	L_A(B)[-1] \to	{\rm RHom}(A, B) \otimes_k A \xrightarrow{ev} B \to L_A(B)
\]
\[
	R_A(B) \to A \xrightarrow{coev} {\rm RHom}(A, B)^* \otimes B \to R_A(B)[1]
\]

which can be taken to define $L_A(B)$, $R_A(B)$. We have the standard result.
\begin{prop}[Gorodentsev, \cite{Rud}]\label{braid} Let $(E, F)$ be an exceptional pair in $\T$. The operations $L, R$ descend to an action on the set of exceptional pairs 
	\begin{align*}
		R:& (E, F) \mapsto (F, R_E(F))\\
		L:& (E, F) \mapsto (L_E(F), E).
	\end{align*}
	Furthermore, $R, L$ are inverses in that $L \circ R (E, F) \simeq (E, F)$ and $R \circ L (E, F) \simeq (E, F)$.
\end{prop}

\subsection*{The equivalence}
Now fix $k$ algebraically closed and $f_\lambda = XY(X-Y)(X-\lambda Y)$ in $S = k[X,Y]$ with $\lambda \neq 0, 1$ in $k$. Let $R = S/f_\lambda$, with $\mathfrak{m} = R_{\geq 1}$. We have $a=d-2=2$. Writing $f_\lambda = l_1l_2l_3l_4$, we have natural matrix factorizations $(l_i, f_\lambda / l_i)$ which present the MCM module $L_i = R/l_i \cong k[z_i]$. Next, since $R$ is Cohen-Macaulay, $\mathfrak{m}$ contains a non-zero divisor and $\mathfrak{m}^n$ has depth $1$ for all $n \geq 1$, and so are MCM modules. Now, the modules $L_i$ have particular simple complete resolutions
\[
	\cdots \leftarrow R(3) \xleftarrow{f_\lambda/l_i} R \xleftarrow{l_i} R(-1) \xleftarrow{f_\lambda/l_i} R(-4) \leftarrow \cdots
\]
Since $\mathfrak{m} = \Omega^1_R(k) = k^{st}[-1]$, using Serre duality we note that the Tate cohomology groups between $\mathfrak{m}(r)$ and $L_i(s)$ are easily computed for any $r,s \in \Z$. We will need: 
\begin{lem}\label{strong} We have 
	\begin{align*}
		&{\rm dim}_k\ \underline{{\rm Ext}}^n_{gr R}(\mathfrak{m}(1), L_i) = \begin{cases}1 & n = 0\\ 0 & n \neq 0\end{cases} \\
		&{\rm dim}_k\ \underline{{\rm Ext}}^n_{gr R}(\mathfrak{m}(2), L_i) = \begin{cases} 1 & n = 1 \\ 0 & n \neq 1. \end{cases}
	\end{align*}
\end{lem}

We will give a proof of the following theorem, which in this case follows standard lines. 
\begin{thm}[Buchweitz-Iyama-Yamaura]\label{tilting} The collection $\left(\mathfrak{m}(1), \mathfrak{m}^2(2), L_1, L_2, L_3, L_4 \right)$ forms a full strong exceptional collection in $\underline{{\rm MCM}}(gr R)$. 
\end{thm}

\begin{proof} We use Orlov's theorem with cutoff $i = -a = -2$. Hence there is a fully faithful embedding
	\[
		\Phi_{-2}: \underline{{\rm MCM}}(gr R) \to {\rm D}^b(gr_{\geq -2} R)
	\]
	which is right-inverse to Buchweitz's stabilization functor ${\rm st}: {\rm D}^b(gr R) \to {\rm D}_{sg}(gr R) \cong \underline{{\rm MCM}}(gr R)$, and a semiorthogonal decomposition involving $Z = V(f_\lambda) \subset \mathbb{P}^1$ given by
	\[
		\Phi_{-2}\left(\underline{\rm MCM}(gr R) \right) = \left\langle k(2), k(1), {\rm R\Gamma}_{\geq 0}({\rm D}^b(Z)) \right\rangle.
	\]
	The category ${\rm D}^b(Z)$ is semisimple generated by the skyscraper sheaves of $4$ points, and we have ${\rm R\Gamma}_{\geq 0}({\rm D}^b(Z)) = \langle L_1, \dots, L_4 \rangle$. This gives a full exceptional sequence $\langle k(2), k(1), L_1, \dots, L_4\rangle$ which is however not strong, and so we mutate it. Now, the extension
	\[
		\xi: \mathfrak{m}/\mathfrak{m}^2 \to R/\mathfrak{m}^2 \to k \to \mathfrak{m}/\mathfrak{m}^2[1] 
	\] agrees with the universal extension
	\[
		{\rm Ext}^1(k, k(-1))^* \otimes_k k(-1) \to R_{k}(k(-1)) \to k \xrightarrow{coev} {\rm Ext}^1(k, k(-1))^{*} \otimes_k k(-1)[1].
	\]
One has ${\rm RHom}(k, k(-1)) = {\rm Ext}^1_{gr R}(k, k(-1))[-1]$, and so we have the right mutation $R_{k}(k(-1)) = R/\mathfrak{m}^2$, and similarly $R_{k(2)}\left(k(1)\right) = \left(R/\mathfrak{m}^2\right)(2)$. After mutating and desuspending the first two terms, we obtain the resulting exceptional collection
\[
	\left\langle k(1)[-1], \left(R/\mathfrak{m}^2\right)(2)[-1], L_1, \dots, L_4\right\rangle
\]
of $\Phi_{-2}\left(\underline{{\rm MCM}}(gr R) \right)$, which descends to the full exceptional collection
\[
	\left\langle \mathfrak{m}(1), \mathfrak{m}^2(2), L_1, \dots, L_4\right\rangle
\]
in $\underline{{\rm MCM}}(gr R)$. Direct calculations with lemma \ref{strong} and the extension $\xi$ shows that this collection is strong.
\end{proof}
Setting $U = \mathfrak{m}(1) \oplus \mathfrak{m}^2(2) \oplus \left( \bigoplus_{i=1}^4 L_i \right)$, we obtain that $U$ is a tilting object for $\underline{{\rm MCM}}(gr R)$. We can calculate the stable endomorphism ring.
\begin{prop}[Buchweitz-Iyama-Yamaura] We have algebra isomorphisms $\underline{{\rm End}}_{gr R}(U) = {\rm End}_{gr R}(U) \cong Sq(2,2,2,2; \lambda)$.
\end{prop}
\begin{proof} Consider the following morphisms
	\[
		\xymatrix@C2pc@R0.3pc{	&& L_1 \\ 
				&& L_2 \\
				\mathfrak{m}(1) \ar@<.5ex>[r]^{X} \ar@<-.5ex>[r]_{Y} & \mathfrak{m}^2(2)  \ar[ruu]^{q_1} \ar[ru]_{q_2} \ar[rd]^{q_3} \ar[rdd]_{q_4} \\
				&& L_3 \\
				&& L_4 \\
}
	\]
	with $q_i: \mathfrak{m}^2(2) \to L_i$ induced by $X,Y \mapsto \overline{X}, \overline{Y} \in R/l_i = L_i$. These are easily seen to generate the endomorphism algebra and satisfy the squid relations, thus showing ${\rm End}_{gr R}(U) \cong Sq(2,2,2,2; \lambda)$. Lastly, a simple dimension count shows that ${\rm End}_{gr R}(U) = \underline{\rm End}_{gr R}(U)$.
\end{proof}
Thinking ahead, we will normalize our tilting object by using $T = U(-3)[1]$ instead of $U$, which has the same endomorphism algebra. Using $T$, and making use of \cite[example 4.4]{LM2}, we obtain:

\begin{cor}\label{wpl} For $\X = \mathbb{P}^1(2,2,2,2; \lambda)$, we have equivalences of triangulated categories
	\[
		\underline{\rm MCM}(gr R) \cong {\rm D}^b(\Lambda) \cong {\rm D}^b(\X).
	\]
The composed equivalence sends the full strong exceptional collection 
	\[
		\left(k^{st}(-2), (R/\mathfrak{m}^2)^{st}(-1), L_1(-3)[1], \dots, L_4(-3)[1]\right)
	\]
	to 
	\[
		\left( \mathcal{O}, \mathcal{O}(\vec{c}), S_{1,0}, \dots, S_{4, 0} \right).
	\]
\end{cor}

As alluded to in the introduction, we also have:
\begin{cor} Let $\widehat{R}$ be the completion of $R$ at $\mathfrak{m}$. Every MCM $\widehat{R}$-module $M$ is the completion of a graded MCM $R$-module. 
\end{cor}
\begin{proof} By \cite[proposition 1.5]{KMVdB}, the completion functor $\widehat{(-)}: \underline{{\rm MCM}}(gr R) \to \underline{{\rm MCM}}(\widehat{R})$ identifies with the universal morphism to the triangulated hull of the orbit category $\underline{{\rm MCM}}(gr R)/(1)$. Since $(4) = [2]$, the functor $(1)$ moves away from the hereditary heart ${\rm coh}(\X) \subset \underline{{\rm MCM}}(gr R)$, and so the completion functor is essentially surjective by results of Keller \cite{Ke}. 
\end{proof}
\section{Betti tables and cohomology tables}
We can calculate graded Betti numbers as $\beta_{i,j}(M) = {\rm dim}_k\ \underline{\rm Ext}_{gr R}^i(M, k^{st}(-j))$, hence as dimension of corresponding ${\rm Ext}$-spaces in ${\rm D}^b(\mathcal{X)}$. Due to the periodicity $k^{st}(-j-4) = k^{st}(-j)[-2]$, it suffices to compute the images of \[
	k^{st}, k^{st}(-1), k^{st}(-2), k^{st}(-3).
\]
The Serre functor $\mathbb{S}_R(M) = M(2)$ is sent to the Serre functor $\mathbb{S}_\X(\F) = \F \otimes \omega_{\X}[1]$. The periodicity identity $(4) = [2]$ corresponds to the fact that $\omega_{\X}$ is $2$-torsion. Keeping this in mind, we will prove the following:

\begin{thm}\label{residuefields} Under the above equivalence, we have
	\begin{align*}
		& k^{st} \ \ \ \ \ \ \mapsto \omega_{\X}[1] \\
		& k^{st}(-1) \mapsto  \mathcal{O}(-\vec{c})[1] \\
		& k^{st}(-2) \mapsto  \mathcal{O}\\
		& k^{st}(-3) \mapsto  \mathcal{O}(-\vec{c}) \otimes \omega_{\X}.
	\end{align*}
\end{thm}
\begin{proof} By corollary \ref{wpl} we already have $k^{st}(-2) \mapsto \mathcal{O}$, and so $k^{st} = \mathbb{S}_R\left(k^{st}(-2)\right) \mapsto \mathbb{S}_{\X}\left(\mathcal{O}\right) = \omega_{\X}[1]$. Let $\F_{k^{st}(-1)}$ correspond to $k^{st}(-1)$. We can obtain the exceptional pair $(k^{st}(-2), \left(R/\mathfrak{m}^2\right)^{st}(-1))$ as the right mutation of
	\[
		R: (k^{st}(-1), k^{st}(-2)) \mapsto \left(k^{st}(-2), R_{k^{st}(-1)}(k^{st}(-2)) \right)
	\]
	and so we can obtain $(\mathcal{O}, \mathcal{O}(\vec{c}))$ as the right mutation of 
	\[
		R: (\F_{k^{st}(-1)}, \mathcal{O}) \mapsto \left(\mathcal{O}, R_{\F_{k^{st}(-1)}}(\mathcal{O}) \right)
	\]
	By \ref{braid} we can recover $\F_{k^{st}(-1)}$ as the left mutation $\F_{k^{st}(-1)} \cong L_{\mathcal{O}} \left( R_{\F_{k^{st}(-1)}}(\mathcal{O})\right) \cong L_{\mathcal{O}}(\mathcal{O}(\vec{c}))$, calculated by the distinguished triangle
	\[
		L_{\mathcal{O}}(\mathcal{O}(\vec{c}))[-1] \to {\rm RHom}_\X(\mathcal{O}, \mathcal{O}(\vec{c})) \otimes_k \mathcal{O} \xrightarrow{ev} \mathcal{O}(\vec{c}) \to L_{\mathcal{O}}(\mathcal{O}(\vec{c})).
	\]

	Finally we have a fully faithful embedding ${\rm D}^b(\mathbb{P}^1) \hookrightarrow {\rm D}^b(\X)$ sending $\mathcal{O}(n)$ to $\mathcal{O}(n\vec{c})$. Calculating in the former one sees that left mutation gives $\F_{k^{st}(-1)} = L_{\mathcal{O}}(\mathcal{O}(\vec{c})) = \mathcal{O}(-\vec{c})[1]$, and so $k^{st}(-1) \mapsto \mathcal{O}(-\vec{c})[1]$ and $k^{st}(-3) \mapsto \mathcal{O}(-\vec{c}) \otimes \omega_{\X}$.
\end{proof}

Given an MCM module $M$, write $\F_M$ for the corresponding complex of coherent sheaves on $\X$.
\begin{cor} We can calculate Betti numbers $\beta_{i,j} = \beta_{i,j}(M)$ as follows:
	\begin{align*}
	& \beta_{i,0} = {\rm dim}_k\ {\rm Ext}^i(\mathcal{F}_{M}, \omega_{\X}[1]) = h^{-i}(\mathcal{F}_{M}) \\
	& \beta_{i,1} = {\rm dim}_k\ {\rm Ext}^i(\mathcal{F}_{M}, \mathcal{O}(-\vec{c})[1]) = h^{-i}(\mathcal{F}_M(\vec{c}) \otimes \omega_{\X})\\
	& \beta_{i,2} = {\rm dim}_k\ {\rm Ext}^i(\mathcal{F}_{M}, \mathcal{O}) = h^{1-i}(\mathcal{F}_{M} \otimes \omega_{\X})\\
	& \beta_{i,3} = {\rm dim}_k\ {\rm Ext}^i(\mathcal{F}_{M}, \mathcal{O}(-\vec{c}) \otimes \omega_{\X}) = h^{1-i}(\mathcal{F}_M(\vec{c})).
	\end{align*}
When $\mathcal{F}_M$ is a coherent sheaf, collecting terms via the periodicity $\beta_{i, j} = \beta_{i+2, j+4}$, the only possible non-trivial Betti numbers for $M$ of the form $\beta_{0, *}, \beta_{1, *}$ are 
\[
	\beta_{0, 0}, \beta_{0, 1}, \beta_{0, 2}, \beta_{0,3}, \beta_{1,2}, \beta_{1,3}, \beta_{1,4}, \beta_{1,5}.
\]
\end{cor}

Since ${\rm coh}(\X)$ is hereditary, indecomposables in ${\rm D}^b(\X)$ are of the form $\mathcal{F}[-n]$ for $\mathcal{F}$ an indecomposable coherent sheaf and $n \in \mathbb{Z}$, and it suffices to work out Betti tables corresponding to coherent sheaves. In this case, the data is best expressed in the following table:
\begin{align*}
	\beta(M) = \begin{pmatrix*}[c]\beta_{0,0} & \beta_{1,2} \\ \beta_{0,1} & \beta_{1,3} \\ \beta_{0,2} & \beta_{1,4} \\ \beta_{0,3} & \beta_{1,5} \\ \end{pmatrix*}
	= \begin{pmatrix*}[l] h^0(\mathcal{F}) & h^0(\mathcal{F} \otimes \omega_{\X}) \\ h^0(\mathcal{F}(\vec{c}) \otimes \omega_{\X}) & h^0(\mathcal{F}(\vec{c})) \\ h^1(\mathcal{F}\otimes \omega_{\X}) & h^1(\mathcal{F}) \\ h^1(\mathcal{F}(\vec{c})) & h^1(\mathcal{F}(\vec{c}) \otimes \omega_{\X}) \\   \end{pmatrix*}
\end{align*}
where $\mathcal{F} = \mathcal{F}_M$. We will refer to the latter table as the cohomology table $\beta(\mathcal{F})$.

\subsection*{Weighted projective lines of genus one}
Let us recall the classification scheme for indecomposable coherent sheaves over a weighted projective line of genus one due to Lenzing and Meltzer \cite{LM}, which closely mirrors Atiyah's classification of sheaves on an elliptic curve. Let $rk, deg: K_0(\X) \to \mathbb{Z}$ be the rank and degree functionals, and let $\mu(\mathcal{F}) = \frac{deg(\mathcal{F})}{rk(\mathcal{F})}$ be the slope of $\mathcal{F}$. For $q \in \mathbb{Q} \cup \{ \infty \}$, let $\mathcal{C}_q$ be the category of semistable sheaves of slope $q$ (see \cite[2.5]{LM}). One can show that indecomposables are semistable, and thus part of $\mathcal{C}_q$ for some $q$. As $deg(\omega_{\X}) = 0$ for $\X$ of genus one, we see that each $\mathcal{C}_q$ is closed under $- \otimes \omega_{\X}$. In \cite{LM}, \cite[chapter 5]{Melt}, Lenzing and Meltzer construct two autoequivalences
\[
	R, S: {\rm D}^b(\X) \to {\rm D}^b(\X)
\]
given by tubular mutations (or spherical twists), which fit into canonical distinguished triangles
\[
	R\F \to \F \xrightarrow{coev} \bigoplus_{j \in \Z_2}{\rm RHom}(\omega_{\X}^{\otimes j}, \F)^* \otimes_k \omega_{\X}^{\otimes j} \to R\F[1]
\]
and (for some fixed choice of $i$)
\[
	\bigoplus_{j \in \Z_2} {\rm RHom}(S_{i, j}, \F) \otimes_k S_{i, j} \xrightarrow{ev} \F \to S\F \to \bigoplus_{j \in \Z_2} {\rm RHom}(S_{i,j}, \F)[1]
\]

These act on rank and degree by
\[
	\begin{pmatrix} rk(R\F) \\ deg(R\F) \end{pmatrix} = \begin{pmatrix} 1 & 1 \\ 0 & 1 \end{pmatrix} \begin{pmatrix} rk(\F) \\ deg(\F) \end{pmatrix}
\]
\[
	\begin{pmatrix} rk(S\F) \\ deg(S\F) \end{pmatrix} = \begin{pmatrix} 1 & 0 \\ 1 & 1 \end{pmatrix} \begin{pmatrix} rk(\F) \\ deg(\F) \end{pmatrix} 
\]

and restrict to equivalences
\[
	R: \mathcal{C}_q \xrightarrow{\cong} \mathcal{C}_{\frac{q}{q+1}}
\]
\[
	S: \mathcal{C}_q \xrightarrow{\cong} \mathcal{C}_{q+1}
\]
for $q \neq \infty$, and
\[
	R: \mathcal{C}_{\infty} \xrightarrow{\cong} \mathcal{C}_1.
\]

By a known result, the transformations $R: q \mapsto \frac{q}{q+1}$ and $S: q \mapsto q+1$ defines an action of the free semigroup on two words $F\{R,S\}$ on the positive rationals $\mathbb{Q}_{+}$, which is free and transitive with single generator $1$. Writing $q \in \mathbb{Q}_{+}$ as $w_q(R, S)\cdot 1$ for some unique word $w_q(R, S)$, one deduces the existence of an autoequivalence
\[
	w_q(R, S) \circ R: {\rm D}^b(\X) \xrightarrow{\cong} {\rm D}^b(\X)
\]
restricting to $\mathcal{C}_{\infty} \xrightarrow{\cong} \mathcal{C}_q$ for any $q > 0$, then extended to any $q \leq 0$ by composing with sufficient powers of $S^{-1}$. Since the category $\mathcal{C}_{\infty}$ consists of all skyscraper sheaves, the category $\mathcal{C}_q$ is serial for any $q$, with simples given by the stable sheaves. Since the type of $\X$ is $(2,2,2,2)$, the Auslander-Reiten quiver of $\mathcal{C}_q$ breaks down into components which are rank one tubes indexed by the ordinary points of $\X$, as well as $4$ tubes of rank two indexed by the exceptional points. These correspond to indecomposables for which $\F \otimes \omega_{\X} \cong \F$ and $\F \otimes \omega_{\X} \ncong \F$, respectively.

For computing morphism spaces, we have the following well-known results.
\begin{lem}[lemma 4.1, \cite{LM}]\label{slopevanishing} Let $\F, \G$ be semistable sheaves of slopes $q, q'$.
	\begin{enumerate}
		\item If $q > q'$, then ${\rm Hom}(\F, \G) = 0$.
		\item If $q < q'$, then ${\rm Ext}^1(\F, \G) = 0$.
	\end{enumerate}
\end{lem}

\begin{prop}[Weighted Riemann-Roch, \cite{LM}]\label{RR}We have 
	\[
		\chi(\F, \G) + \chi(\F, \G \otimes \omega_{\X}) = \begin{vmatrix} rk(\F) & rk(\G) \\ deg(\F) & deg(\G) \end{vmatrix}.
	\]
	In particular, we get
	\[
		\chi(\F) + \chi(\F \otimes \omega_{\X}) = deg(\F).
	\]
\end{prop}

\subsection*{The shift autoequivalence}
The autoequivalence $M \mapsto M(1)$ on $\underline{\rm MCM}(gr R)$ induces an autoequivalence $\F_M \mapsto \F_M\{1\}$ on ${\rm D}^b(\X)$. Since shifting the grading acts by translation on Betti tables, it suffices to compute one Betti table in the orbit $\{M(n)\}_{n \in \Z}$, or equivalently one cohomology table in the orbit $\{\F_M\{n\} \}_{n \in \Z}$. First, we calculate the effect of $(1)$ on rank and degree. Observe that these pull back to functionals $rk, deg: K_0(\underline{\rm MCM}(gr R)) \to \Z$.

\begin{lem}\label{rkdeg} For any graded MCM module M, we have 
	\begin{align*}
		deg(M)  &= \chi(M, k^{st}) - \chi(M, k^{st}(-2)) \\
			&= \sum_{i \in \Z}(-1)^{i} \beta_{i, 0} - \sum_{i \in \Z}(-1)^{i} \beta_{i, 2},	\\
			rk(M) 	&= \frac{1}{2} \chi(M, k^{st}(-1)\oplus k^{st}(-2)) - \frac{1}{2} \chi( M, k^{st} \oplus k^{st}(-3)) \\
			&= \frac{1}{2} \sum_{i \in \Z} (-1)^{i}(\beta_{i, 1} + \beta_{i, 2}) - \frac{1}{2} \sum_{i \in \Z} (-1)^{i}(\beta_{i, 0} + \beta_{i,3}).
	\end{align*}
When M corresponds to a coherent sheaf, this simplifies to
\begin{align*}
	deg(M) &= (\beta_{0, 0} + \beta_{1,2}) - (\beta_{0,2} + \beta_{1,4}), \\
	rk(M) &= \frac{1}{2}(\beta_{0,1} + \beta_{0,2} + \beta_{1,3} + \beta_{1,4}) - \frac{1}{2}(\beta_{0,0} + \beta_{0,3} + \beta_{1,2} + \beta_{1,5}).
\end{align*}
\end{lem}
\begin{proof} The formula for $deg(M)$ falls out of the weighted Riemann-Roch theorem via theorem \ref{residuefields}. To deduce the formula for $rk(M)$, we use $deg(\F(\vec{c})) = deg(\F) + 2 rk(\F)$, so that $rk(\F) = \frac{1}{2}(deg(\F(\vec{c})) - deg(\F))$, then collect terms via \ref{residuefields}.
\end{proof}

\begin{prop}\label{matrix} For any graded MCM module $M$, we have 
	\[
		\begin{pmatrix} rk(M(1)) \\ deg(M(1)) \end{pmatrix} = \begin{pmatrix} -1 & -1 \\ 2 & 1 \end{pmatrix} \begin{pmatrix} rk(M) \\ deg(M) \end{pmatrix}.
	\]
\end{prop}
\begin{proof}Writing $\ZZ(M)$ for $\left(rk(M)\ deg(M)\right)^{T}$, we need to verify that 
	\[
		\xymatrix@C5pc{
			K_0 \ar[d]^{(1)}\ar[r]^{\ZZ} & \Z^2 \ar[d]^{\begin{psmallmatrix} -1 & -1 \\ 2 & 1 \end{psmallmatrix}} \\
		K_0 \ar[r]^{\ZZ} & \Z^2 \\
	}
\]
commutes. The full exceptional collection $\left(k^{st}(-1), k^{st}(-2), L_1(-3)[1], \dots, L_4(-3)[1] \right)$ descends to a $\Z$-basis of $K_0$, on which this can be checked. We have
\begin{align*}
&\ZZ\left(k^{st}(-2)\right) = \ZZ\left(\mathcal{O}\right) = (1\ 0)^T\\
	&\ZZ\left(k^{st}(-1)\right) = \ZZ\left(\mathcal{O}(-\vec{c})[1]\right) = (-1\ 2)^T\\
	&\ZZ\left(k^{st}\right) = \ZZ\left(\omega_{\X}[1]\right) = (-1\ 0)^T\\
	\\
	&\ZZ\left(L_i(-3)[1]\right) = \ZZ\left(S_{i, 0} \right) =  (0\ 1)^T\\
	&\ZZ\left(L_i(-2)[1]\right) = (-1\ 1)^T.
\end{align*}

The last line is calculated by lemma \ref{rkdeg}, and so the above diagram commutes.
\end{proof} 
\begin{rem} The above matrix has order $4$ in ${\rm SL}(2, \Z)$, corresponding to the identity $(4) = [2]$ on $\underline{{\rm MCM}}(gr R)$. \end{rem}
Writing $r, d$ for rank and degree, it is not hard to find a fundamental domain for the action of $(1)$ on the $(r,d)$ lattice\footnote{We will ignore $(0,0)$ since any indecomposable with $rk(M) = 0 = deg(M)$ must be zero.}. We will take our fundamental domain as the union of $3$ regions, as shown in figure $1$.
\begin{figure}\label{fig1}
   \begin{minipage}[r]{4cm}
       \begin{tikzpicture}
	    \draw[gray!50, thin, step=0.3] (-2,-2) grid (2,2);
	    \draw[very thick,->] (-2,0) -- (2.2,0) node[right] {$r$};
	    \draw[very thick,->] (0,-2) -- (0,2.2) node[above] {$d$};

	    \foreach \x in {-2,...,2} \draw (\x,0.03) -- (\x,-0.03) node[below] {\tiny\x};
	    \foreach \y in {-2,...,2} \draw (-0.03,\y) -- (0.03,\y) node[right] {\tiny\y};

	    \fill[blue!50!cyan,opacity=0.3] (2,0) -- (0,2) -- (2,2);
	    \fill[blue!50!cyan,opacity=0.3] (0,0) -- (0,2) -- (2,0);
	    \fill[blue!50!cyan,opacity=0.3] (0,0) -- (1,-2) -- (0,-2);

	    \draw (-1,2) -- node[below,sloped] {} (1,-2);
       \end{tikzpicture}
   \end{minipage}%
   \begin{minipage}[c]{\textwidth-7cm}
\begin{align}
	& r \geq 0,\ d > 0 \\
	& r > 0,\ d = 0 \\ 
	& r > 0,\ d < -2r.
\end{align}
   \end{minipage}
   \caption{Fundamental domain for the action of $(1)$ on $(r,d)$.}
\end{figure}
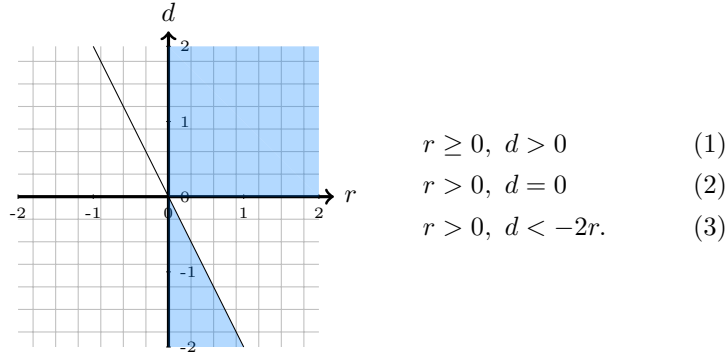
Note that since $r \geq 0$ throughout, each pair $(r,d)$ is realized by a coherent sheaf and we need not consider complexes. Another reason for this choice is to maximize vanishing patterns in the cohomology table
\[
	\beta(\F) = \begin{pmatrix*}[l] h^0(\mathcal{F}) & h^0(\mathcal{F} \otimes \omega_{\X}) \\ h^0(\mathcal{F}(\vec{c}) \otimes \omega_{\X}) & h^0(\mathcal{F}(\vec{c})) \\ h^1(\mathcal{F}\otimes \omega_{\X}) & h^1(\mathcal{F}) \\ h^1(\mathcal{F}(\vec{c})) & h^1(\mathcal{F}(\vec{c}) \otimes \omega_{\X}) \\   \end{pmatrix*}
\]

\begin{lem}\label{h-vanishing} Let $\F$ be an indecomposable coherent sheaf.
	\begin{enumerate}
		\item For $\F$ in region $(1)$, $h^1(\F) = h^1(\F \otimes \omega_{\X}) = h^1(\F(\vec{c})) = h^1(\F(\vec{c})\otimes \omega_{\X}) = 0.$
	\item For $\F$ in region $(3)$, $h^0(\F) = h^0(\F \otimes \omega_{\X}) = h^0(\F(\vec{c})) = h^0(\F(\vec{c})\otimes \omega_{\X}) = 0.$	\end{enumerate}
\end{lem}
\begin{proof} These follow from Lemma \ref{slopevanishing} by slope arguments, using the formula $$\mu(\F \otimes \mathcal{L}) = \mu(\F) + deg(\mathcal{L})$$ for a line bundle $\mathcal{L}$. \end{proof}
This lemma reduces calculations in regions $(1), (3)$ to computing Euler characteristics, and region $(2)$ can be dealt with by hand.

\section{Cohomology tables of indecomposable coherent sheaves}
\subsection*{Cohomology tables for rank one tubes}
We are now in a position to compute the cohomology tables of indecomposable sheaves. We will list the corresponding possible Betti tables of matrix factorizations in a later section, under a different normalization. We begin with indecomposables living in rank one tubes, or equivalently which satisfy $\F \otimes \omega_{\X} \cong \F$.

\begin{thm}\label{rankone} Let $(r,d)$ be in the fundamental domain with $q = \frac{d}{r}$. Consider $\F$ with $\F \otimes \omega_{\X} \cong \F$ and $(rk(\F), deg(\F)) = (r,d)$. An indecomposable such $\F$ exists if and only if $gcd(r,d)$ is even, in which case $\frac{|gcd(r,d)|}{2}$ gives the length of $\F$ in $\mathcal{C}_q$. The cohomology table $\beta(\F)$ is then given as:
\[
	\begin{tabular}{|c|c|c|}
	\hline
	$r \geq 0,\ d > 0$ & $r > 0,\ d = 0$ & $r > 0,\ d < -2r$ \\
	\hline
	&& \\
	$\begin{pmatrix} \frac{d}{2} & \frac{d}{2} \\ \frac{d}{2} + r & \frac{d}{2} + r \\ 0 & 0 \\ 0 & 0 \\ \end{pmatrix}$
	& $\begin{pmatrix} 0 & 0 \\ r & r \\ 0 & 0 \\ 0 & 0 \\ \end{pmatrix}$
	& $\begin{pmatrix} 0 & 0 \\ 0 & 0 \\ -\frac{d}{2} & -\frac{d}{2} \\ -\frac{d}{2} - r & -\frac{d}{2} - r \end{pmatrix}$ \\
	&& \\
	\hline
\end{tabular}
\]
\end{thm}
\begin{proof} Let $\Phi_{q, \infty}: \mathcal{C}_q \xrightarrow{\cong} \mathcal{C}_\infty$ be the Lenzing-Meltzer autoequivalence, which is a composite of $R^{\pm}, S^{\pm}$ and thus acts on $(r,d)$ by $SL(2, \Z)$ transformations. The torsion sheaf $\Phi_{q, \infty}(\F)$ has type $(r', d') = (0, d')$ with $d' = |gcd(r', d')| = |gcd(r,d)|$. Since $\F \otimes \omega_\X \cong \F$, $\Phi_{q, \infty}(\F)$ is supported on an ordinary point and $d'$ must be even, and is simple in $\mathcal{C}_\infty$ precisely when $d' = 2$. This proves the claim except for the shape of $\beta(\F)$.
	
Now, $\F \mapsto \F \otimes \omega_{\X}$ acts by column change on cohomology tables, and so $\beta(\F)$ is symmetrical. By Riemann-Roch we have $2 \cdot \chi(\F) = d$ and $2 \cdot \chi(\F(\vec{c})) = d + 2r$. Combining this with Lemma $\ref{h-vanishing}$ determines tables in region $(1), (3)$. Next, let $\F$ be in region $(2)$. Then both $\mathcal{O}$, $\F$ are in $\mathcal{C}_0$, and $\mathcal{O}$ lives in a rank two tube. Applying $\Phi_{0, \infty}$ to both sends them to skyscraper sheaves with disjoint support, and thus ${\rm Ext}^*(\mathcal{O}, \F) = 0$. An application of Lemma \ref{h-vanishing} and Riemann-Roch as above determines $\beta(\F)$ in $(2)$.  
\end{proof}

\subsection*{Cohomology tables for rank two tubes}
We now study indecomposable sheaves $\F$ with $\F \otimes \omega_{\X} \ncong \F$. We begin with some generalities.

\begin{prop}\label{ranktwoprop} Let $(r,d)$ be in the fundamental domain and $q = \frac{d}{r}$. The following hold:
	\begin{enumerate}
		\item For any $(r,d)$, there is an indecomposable $\F$ with $(rk(\F), deg(\F)) = (r,d)$ and $\F \otimes \omega_\X \ncong \F$.
		\item Any such indecomposable $\F$ has length $|gcd(r,d)|$ in $\mathcal{C}_q$.
		\item There are finitely many indecomposable sheaves of type $(r,d)$ if and only if $gcd(r,d)$ is odd, in which case there are exactly eight.
		\item There is an exceptional sheaf of type $(r,d)$ if and only if $|gcd(r,d)| = 1$, in which case all such indecomposables are exceptional.
		\item When $gcd(r,d)$ is even and $d \neq 0$, $\beta(\F) = \beta(\widetilde{\F})$ where $\widetilde{\F}$ is indecomposable of same rank and degree, and $\widetilde{\F} \otimes \omega_{\X} \cong \widetilde{\F}$.
	\end{enumerate}
\end{prop}
\begin{proof} The first four points follow from the autoequivalence $\Phi_{q, \infty}$ as in the proof of theorem \ref{rankone}. For $(5)$, similarly reduce to skyscraper sheaves. Let $S^{[2n]}$ be an indecomposable torsion sheaf supported over the exceptional point $x_i$ of degree $2n$. Then $[S^{[2n]}]$ has height $2n$ in its tube, and computing Grothendieck classes gives $[S^{[2n]}] = n [S_{i, 0}] + n[S_{i, 1}]$. In particular the ``ordinary'' torsion sheaf $S_{x}$ for $x = x_i$ has degree $2$, and we have $[S_x] = [S_{i,0}] + [S_{i,1}]$. From the presentation $0 \to \mathcal{O}(-\vec{c}) \to \mathcal{O} \to S_x \to 0$ we see that $[S_x] = [S_{x'}]$ for any ordinary point $x'$, and so $[S^{[2n]}] = n[S_x] = n[S_{x'}] = [S_{x'}^{[n]}]$ where $S_{x'}^{[n]}$ is a length $n$ indecomposable sheaf supported at $x'$. 
	
	Pulling back through $\Phi_{q, \infty}$, we deduce that for any indecomposable $\F$ with $gcd(r,d)$ even, there is another indecomposable $\widetilde{\F}$ of type $(r,d)$ with $[\widetilde{\F}] = [\F]$ and $\widetilde{\F}\otimes \omega_{\X} \cong \widetilde{\F}$. Since $[\widetilde{\F}] = [\F]$, we have $\chi(\widetilde{\F} \otimes \mathcal{L}) = \chi(\F \otimes \mathcal{L})$ for any line bundle $\mathcal{L}$, and outside of the case $d = 0$, those values determine the cohomology table, hence $\beta(\F) = \beta(\widetilde{\F})$.
\end{proof}

The remainder of the section will be aimed at proving the next theorem.
\begin{thm}\label{ranktwothm} Let $(r,d)$ be in the fundamental domain. The $\beta(\F)$ of indecomposables of type $(r,d)$ satisfying $\F \otimes \omega_{\X} \ncong \F$ are listed as follows: 
\[
\resizebox{\columnwidth}{!}{%
\begin{tabular}{|c|c|c|}
	\hline
	$(r,d)$ & $r \geq 0,\ d > 0$ & $r > 0,\ d < -2r$ \\
	\hline
	&& \\
$d$ odd & $\begin{pmatrix} \frac{d \pm 1}{2} & \frac{d\mp 1}{2} \\ \frac{d\mp1}{2} + r & \frac{d\pm1}{2} + r \\ 0 & 0 \\ 0 & 0 \\ \end{pmatrix}_4$ & $\begin{pmatrix} 0 & 0 \\ 0 & 0 \\ -\frac{d\pm1}{2} & -\frac{d\mp1}{2} \\ -\frac{d\mp1}{2} - r & -\frac{d\pm1}{2} - r  \\ \end{pmatrix}_4$ \\
	&& \\   
	\hline
	&& \\ 	
	(odd, even) & $\begin{pmatrix} \frac{d}{2} \pm 1 & \frac{d}{2} \mp 1 \\ \frac{d}{2} \mp 1 + r  & \frac{d}{2} \pm 1 + r \\ 0 & 0 \\ 0 & 0 \\ \end{pmatrix}_1  \begin{pmatrix}\frac{d}{2} & \frac{d}{2} \\ \frac{d}{2} + r & \frac{d}{2} + r \\ 0 & 0 \\ 0 & 0 \\ \end{pmatrix}_6$ & $\begin{pmatrix} 0 & 0 \\ 0 & 0 \\ -\frac{d}{2} \pm 1 & -\frac{d}{2} \mp 1 \\ -\frac{d}{2} \mp 1 - r & -\frac{d}{2} \pm1 - r \\ \end{pmatrix}_1 \begin{pmatrix} 0 & 0 \\ 0 & 0 \\ -\frac{d}{2} & -\frac{d}{2} \\ -\frac{d}{2} - r & -\frac{d}{2} - r \\ \end{pmatrix}_6$ \\
	&& \\
	\hline
	&& \\
	(even, even) & $\begin{pmatrix} \frac{d}{2} & \frac{d}{2} \\ \frac{d}{2} + r & \frac{d}{2} + r  \\ 0 & 0 \\ 0 & 0 \\ \end{pmatrix}_8$ & $\begin{pmatrix} 0 & 0 \\ 0 & 0 \\ -\frac{d}{2} & -\frac{d}{2} \\ -\frac{d}{2} - r & -\frac{d}{2} - r \\ \end{pmatrix}_8$  \\
	&& \\
	\hline
\end{tabular}%
}
\]
\[
\begin{tabular}{|l|c|c|}
	\hline
	$r > 0,\ d = 0$ & Tube contains $\mathcal{O}_{\X}$ & Tube does not contain $\mathcal{O}_{\X}$ \\
	\hline
	&& \\
	r odd & $\begin{pmatrix} 1 & 0 \\ r-1 & r+1 \\ 1 & 0 \\ 0 & 0 \\ \end{pmatrix}_1 \begin{pmatrix} 0 & 1 \\ r+1 & r-1 \\ 0 & 1 \\ 0 & 0 \\ \end{pmatrix}_1$  & $\begin{pmatrix} 0 & 0 \\ r & r \\ 0 & 0 \\ 0 & 0 \\ \end{pmatrix}_6$ \\
	&& \\ 
	\hline
	&& \\
	r even & $\begin{pmatrix} 1 & 0 \\ r & r \\ 0 & 1 \\ 0 & 0 \\ \end{pmatrix}_1$ $\begin{pmatrix} 0 & 1 \\ r & r \\ 1 & 0 \\ 0 & 0 \\ \end{pmatrix}_1$ & $\begin{pmatrix} 0 & 0 \\ r & r \\ 0 & 0 \\ 0 & 0 \\ \end{pmatrix}_6$ \\
	&& \\ 
	\hline
\end{tabular}
\]
The subscript counts the number of indecomposables satisfying $\F \otimes \omega_{\X} \ncong \F$  with given cohomology table.
\end{thm}
{\bf The proof strategy}. We will make use of Crawley-Boevey's generalization of Kac's Theorem for weighted projective lines. By the previous proposition, indecomposables with $gcd(r,d)$ odd correspond to real roots of the associated root system, which are enumerated in a standard basis for $K_0$. Going through the list, one tabulates all triples $(rk(\F), deg(\F), \chi(\F))$ coming from real roots $[\F]$, and this triple completely determines $\beta(\F)$ in regions $(1), (3)$. The region $(2)$ is then dealt with by hand. We first recall the needed notions, following \cite{CB, Schiff}. 

\subsection*{Kac's theorem, after Schiffmann-Crawley-Boevey}Let $\X = \mathbb{P}^1({\bf p}, \pmb{\lambda})$ be a general weighted projective line for now, and let $T_{can} = \bigoplus_{\vec{0} \leq \vec{x} \leq \vec{c}} \mathcal{O}(\vec{x})$ be the canonical tilting object with endomorphism algebra $\Lambda_{can} = kQ/I$ with quiver $Q$:
\[
	\xymatrix@R.5pc{ & \vec{x_1} \ar[r] & 2\vec{x_1} \ar[r] & \dots \ar[r] & (p_1-1)\vec{x_1} \ar[rdd] & \\
		& \vec{x_2} \ar[r] & 2\vec{x_2} \ar[r] & \dots \ar[r] & (p_2-1)\vec{x_2} \ar[rd] & \\	
	 \vec{0} \ar[ruu] \ar[ru] \ar[rd] &\vdots & \vdots & \vdots & \vdots & \vec{c} \\
	 & \vec{x}_n \ar[r] & 2\vec{x}_n \ar[r] & \dots \ar[r] & (p_n -1)\vec{x}_n \ar[ru] & \\	
	}
\]
Let $Q'$ be the tree subquiver corresponding to $T' = \bigoplus_{\vec{0} \leq \vec{x} < \vec{c}} \mathcal{O}(\vec{x})$, which we label differently as
\[
	\xymatrix@R.5pc{ & 1,1 \ar[r] & 1,2 \ar[r] & \dots \ar[r] & 1, p_1-1 \\
		& 2,1 \ar[r] & 2,2 \ar[r] & \dots \ar[r] & 2, p_2-1 \\	
	 0 \ar[ruu] \ar[ru] \ar[rd] &\vdots&\vdots&\vdots&\vdots \\
	 & n,1 \ar[r] & n,2 \ar[r] & \dots \ar[r] & n, p_n -1 \\	
	}
\]
	Let $\mathfrak{g}$ be its associated Kac-Moody algebra with root system $\Gamma$, with simple roots $\epsilon_0$, $\epsilon_{i,j}$, and let $L\mathfrak{g} = \mathfrak{g}[t, t^{-1}]$ its loop algebra with root system $\widehat{\Gamma} = \Z \delta \oplus \Gamma$, with $(\delta, -) = 0$. The derived equivalence
	\[
		{\rm D}^b(\X) \cong {\rm D}^b(\Lambda_{can})
	\]
	sends $\{\mathcal{O}, S_{i, j}\}_{j \neq 0}$ to the simple modules $S(0), S(i,j)$ supported over $Q'$. This identifies the summand $\Z[\mathcal{O}] \oplus (\bigoplus_{i, j} \Z[S_{i,j}])$ of $K_0(\X)$ with $\Gamma$, sending the symmetrized Euler form with the Weyl-invariant symmetric bilinear form on $\Gamma$. As Schiffmann then shows \cite{Schiff}, this extends to a full isomorphism $K_0(\X) \xrightarrow{\cong} \widehat{\Gamma}$ sending $[S_x]$ to $\delta$. The induced positive cone given by classes of coherent sheaves on $\widehat{\Gamma}$ is given by nonnegative combinations of
	\[
		\epsilon_0,\ \epsilon_0 + n\delta,\ \epsilon_{i,j},\ \delta - \sum_{j \neq 0} \epsilon_{i,j}, \quad n \in \Z
	\]
	with $[\mathcal{O}(n\vec{c})] \mapsto \epsilon_0 + n\delta$ and $[S_{i,0}] \mapsto \delta - \sum_{j \neq 0} \epsilon_{i,j}$. A version of Kac's Theorem then holds for coherent sheaves on $\X$, which we only state in a weak form:
	\begin{prop}[Crawley-Boevey, \cite{CB}] The isomorphism $K_0(\X) \xrightarrow{\cong} \widehat{\Gamma}$ induces a bijection between Grothendieck classes of indecomposable coherent sheaves and the positive roots of $L\mathfrak{g}$.
		\begin{enumerate}
			\item When $\beta$ is a positive real root, then there is a unique indecomposable $\F$ such that $[\F] \mapsto \beta$.
			\item When $\beta$ is a positive imaginary root, then there are infinitely many indecomposables $\F$ for which $[\F] \mapsto \beta$.
		\end{enumerate}
	\end{prop}
	This extends naturally to indecomposables in ${\rm D}^b(\X)$ by considering all roots and complexes up to $\F \mapsto \F[2]$. Letting $\Delta$ be the set of roots of $\mathfrak{g}$, the roots of $L\mathfrak{g}$ are given by $\{\alpha + n\delta\ |\ \alpha \in \Delta,\ n \in \Z\}$, and the real roots are those of the form $\alpha + n\delta$ with $\alpha \in \Delta^{re}$. 

	Coming back to $\X = \mathbb{P}^1(2,2,2,2;\lambda)$, $\mathfrak{g}$ is of affine type $\widetilde{D}_4$ and we write $\epsilon_{i}$ for $\epsilon_{i,1}$. The positive real roots of $\widetilde{D}_4$ are given by solutions $\alpha = \sum_{i=0}^4 \alpha_i \epsilon_i$ to
	\[
		q(\alpha) = \left(\alpha_0^2 + \alpha_1^2 + \dots + \alpha_4^2 \right) - \left(\alpha_0\alpha_1 + \dots + \alpha_0\alpha_4 \right) = 1
	\]
	with $\alpha_i \in \Z_{\geq 0}$. Writing $q(\alpha) = \sum_{i = 1}^4 \left(\alpha_i - \frac{1}{2}\alpha_0 \right)^2 \geq 0$, one sees that the solutions are as listed in figure $2$.
	\begin{figure}\label{fig2}
\[
	\centerline{
\small{\begin{tabular}{|c|c|c|c|c||c|c|c|}
	\hline
	$\alpha_0$ & $\alpha_1$ & $\alpha_2$ & $\alpha_3$ & $\alpha_4$ &  $r$ & $d$ & $\chi$ \\
	\hline
	2m & m+1 & m & m & m &  2m & 4m+1+2n & 2m+n \\	
	2m & m & m+1 & m & m &  2m & 4m+1+2n & 2m+n \\	
	2m & m & m & m+1 & m &  2m & 4m+1+2n  & 2m+n \\	
	2m & m & m & m & m+1 &  2m & 4m+1+2n  & 2m+n \\	
	&&&&&&& \\ 
	2m+1 & m & m & m & m &  2m+1 & 4m+2n  & 2m+1+n  \\
	&&&&&&& \\ 
	2m+1 & m+1 & m & m & m  & 2m+1 & 4m+1+2n  & 2m+1+n  \\ 
	2m+1 & m & m+1 & m & m  & 2m+1 & 4m+1+2n  & 2m+1+n  \\ 
	2m+1 & m & m & m+1 & m  & 2m+1 & 4m+1+2n  & 2m+1+n  \\ 
	2m+1 & m & m & m & m+1  & 2m+1 & 4m+1+2n  & 2m+1+n  \\ 
	&&&&&&& \\            
	2m+1 & m+1 & m+1 & m & m &  2m+1 & 4m+2+2n  & 2m+1+n  \\ 
	2m+1 & m & m & m+1 & m+1 &  2m+1 & 4m+2+2n  & 2m+1+n  \\ 
	2m+1 & m+1 & m & m+1 & m &  2m+1 & 4m+2+2n  & 2m+1+n  \\ 
	2m+1 & m & m+1 & m & m+1 &  2m+1 & 4m+2+2n  & 2m+1+n  \\ 
	2m+1 & m+1 & m & m & m+1 &  2m+1 & 4m+2+2n  & 2m+1+n  \\ 
	2m+1 & m & m+1 & m+1 & m &  2m+1 & 4m+2+2n  & 2m+1+n  \\ 
	&&&&&&& \\ 
	2m+1 & m & m+1 & m+1 & m+1  & 2m+1 & 4m+3+2n  & 2m+1+n  \\ 
	2m+1 & m+1 & m & m+1 & m+1  & 2m+1 & 4m+3+2n  & 2m+1+n  \\ 
	2m+1 & m+1 & m+1 & m & m+1  & 2m+1 & 4m+3+2n  & 2m+1+n  \\ 
	2m+1 & m+1 & m+1 & m+1 & m  & 2m+1 & 4m+3+2n  & 2m+1+n  \\ 
	&&&&&&& \\ 
	2m+1 & m+1 & m+1 & m+1 & m+1 &  2m+1 & 4m+4+2n  & 2m+1+n  \\ 
	&&&&&&& \\ 
	2m+2 & m & m+1 & m+1 & m+1 &  2m+2 & 4m+3+2n  & 2m+2+n   \\ 
	2m+2 & m+1 & m & m+1 & m+1 &  2m+2 & 4m+3+2n  & 2m+2+n   \\ 
	2m+2 & m+1 & m+1 & m & m+1 &  2m+2 & 4m+3+2n  & 2m+2+n  \\ 
	2m+2 & m+1 & m+1 & m+1 & m &  2m+2 & 4m+3+2n  & 2m+2+n \\ 
	\hline
	\end{tabular}}}
\]
	\caption{Positive real roots of affine $D_4$ and triples $(r, d, \chi)$, where $m \geq 0,\ n \in \Z$.} 
\end{figure}
The last three columns record $(rk(\F), deg(\F), \chi(\F))$ where $\F$ is any indecomposable corresponding to $\beta = \alpha + n\delta$. To see this, note that we must have $[\F] = \alpha_0 [\mathcal{O}] + \sum_{i=1}^4 \alpha_i [S_{i,1}] + n[S_x]$ and that $\chi(S_{i, 1}) = 0$. A general real root then has the form $\beta = \pm \alpha + n \delta$, with $\alpha$ in the above table and $n \in \Z$. Finally, let us record a lemma before moving on to the proof of theorem \ref{ranktwothm}.
\begin{lem}\label{chi} Let $\F$ be an indecomposable with $[\F] \mapsto \beta$ real, with $r \geq 0,\ d \in \Z$. Then the possible values of $\chi(\F)$ only depend on $d$ and are listed below: 
	\[
		\begin{tabular}{|c|c|c|}
			\hline
			$d$ & $\chi$ \\
			\hline
			& \\
			$d$ odd & $\left(\frac{d-1}{2}\right)_4,\ \left(\frac{d+1}{2}\right)_4$ \\ 
			& \\
			\hline
			& \\
			$d$ even & $\left(\frac{d}{2}-1\right)_1,\ \left(\frac{d}{2}\right)_6,\ \left(\frac{d}{2}+1\right)_1$ \\ 
			& \\
			\hline
		\end{tabular}
	\]
	The subscript indicates how many times $\chi$ appears for fixed $(r,d)$. 
\end{lem}
\begin{proof}  This follows by inspection of figure $2$, where the case $r > 0$ corresponds to $\beta = \alpha + n\delta$ for $\alpha$ in the table, and $r = 0$ uses $\pm \alpha$ for $\alpha$ in the first four rows. 
\end{proof}
We can now compute cohomology tables of indecomposables in rank two tubes
\[
	\beta(\F) = \begin{pmatrix*}[l] h^0(\mathcal{F}) & h^0(\mathcal{F} \otimes \omega_{\X}) \\ h^0(\mathcal{F}(\vec{c}) \otimes \omega_{\X}) & h^0(\mathcal{F}(\vec{c})) \\ h^1(\mathcal{F}\otimes \omega_{\X}) & h^1(\mathcal{F}) \\ h^1(\mathcal{F}(\vec{c})) & h^1(\mathcal{F}(\vec{c}) \otimes \omega_{\X}) \\   \end{pmatrix*}
\]
\begin{proof} Let $\F$ be indecomposable with $\F \otimes \omega_{\X} \ncong \F$, of type $(r,d)$. First assume $(r,d)$ in region $(1)$, the case $(3)$ being similar. The bottom half of $\beta(\F)$ vanishes for slope reasons (lemma \ref{h-vanishing}). When $gcd(r,d)$ is even then by proposition \ref{ranktwoprop} the table $\beta(\F)$ is as in theorem \ref{rankone}. When $gcd(r,d)$ is odd, then by \ref{ranktwoprop} and Kac's Theorem the class $[\F]$ must correspond to a real root, and so the possible values of $\chi$ are listed in lemma \ref{chi}, which determines the possible values of $h^0(\F)$. By Riemann-Roch we have
	\[
		h^0(\F) + h^0(\F \otimes \omega_{\X}) = d
	\]
	\[
		h^0(\F(\vec{c})) + h^0(\F(\vec{c}) \otimes \omega_{\X}) = d+2r
	\]
	Now, keeping in mind that $\beta(\F) = \beta(M)$ where $M$ is presented by a matrix factorization, the sum of each column must be equal. From this one sees that the tables must be of the form
\[
\centerline{
	\begin{tabular}{|c|c|}
	\hline
	$(r,d)$ & $r \geq 0,\ d > 0$ \\
	\hline
	& \\
$d$ odd & $\begin{pmatrix} \frac{d \pm 1}{2} & \frac{d \mp 1}{2} \\ \frac{d \mp 1}{2} + r & \frac{d \pm 1}{2} + r \\ 0 & 0 \\ 0 & 0 \\ \end{pmatrix}_4$ \\
	& \\   
	\hline
	& \\ 	
	$d$ even & $\begin{pmatrix} \frac{d}{2} \pm 1 & \frac{d}{2} \mp 1 \\ \frac{d}{2} \mp 1 + r  & \frac{d}{2} \pm 1 + r \\ 0 & 0 \\ 0 & 0 \\ \end{pmatrix}_{1}  \begin{pmatrix}\frac{d}{2} & \frac{d}{2} \\ \frac{d}{2} + r & \frac{d}{2} + r \\ 0 & 0 \\ 0 & 0 \\ \end{pmatrix}_6$ \\
	& \\
	\hline
\end{tabular}}
\]

The same argument determines tables in region $(3)$, so we are left with region $(2)$ where $r > 0,\ d = 0$. We first note that the base of rank two tubes in $\mathcal{C}_{0}$ consists of line bundles of degree $0$, with one distinguished tube having $\{\mathcal{O}, \omega_{\X}\}$ as base. For the other tubes, an appeal to the autoequivalence $\Phi_{0, \infty}: \mathcal{C}_0 \xrightarrow{\cong} \mathcal{C}_\infty$ shows that ${\rm Ext}^*(\mathcal{O}, \F) = 0 = {\rm Ext}^*(\mathcal{O}, \F \otimes \omega_{\X})$. The first row and third row of $\beta(\F)$ vanishes, and slope considerations show vanishing of the fourth row. The table must then be
\[
	\begin{pmatrix} 0 & 0 \\ r & r \\ 0 & 0 \\ 0 & 0 \end{pmatrix}
\]
We are down to $\F$ in the Auslander-Reiten component of $\{\mathcal{O}, \omega_{\X}\}$. Now, $\F$ is uniserial with socle either $\mathcal{O}$ or $\omega_\X$. Assume the first. From the structure of a rank two tube, the simple top of $\F$ is $\omega_{\X}$ when $r$ is even, and $\mathcal{O}$ for $r$ odd. This determines the dimensions of ${\rm Hom}(\mathcal{O}, \F),\ {\rm Hom}(\omega_{\X}, \F),\ {\rm Hom}(\F, \mathcal{O}),\ {\rm Hom}(\F, \omega_{\X})$ as
\begin{align*}
	&{\rm dim}_k \ {\rm Hom}(\mathcal{O}, \F) = 1\\
	&{\rm dim}_k \ {\rm Hom}(\omega_\X, \F) = 0\\
	&{\rm dim}_k \ {\rm Hom}(\F, \mathcal{O}) = \begin{cases} 1 & \textup{r odd} \\ 0 & \textup{r even} \end{cases} \\
	&{\rm dim}_k \ {\rm Hom}(\F, \omega_\X) = \begin{cases} 0 & \textup{r odd} \\ 1 & \textup{r even} \end{cases}  
\end{align*}
and from Serre duality one deduces the shape of the first and third rows. This is enough to determine the tables as  
\[
\centerline{
	\begin{tabular}{|c|c|}
	\hline
	r odd & r even\\
	\hline
	& \\ 
	$\begin{pmatrix} 1 & 0 \\ r-1 & r+1 \\ 1 & 0 \\ 0 & 0 \\ \end{pmatrix}$ & 
	$\begin{pmatrix} 1 & 0 \\ r & r \\ 0 & 1 \\ 0 & 0 \\ \end{pmatrix}$ \\
	& \\ 
	\hline
\end{tabular}}
\]
The case of $\F$ with socle $\omega_{\X}$ is then given by the mirrored table. \end{proof}
\section{Betti tables of indecomposable matrix factorizations}
The previous fundamental domain for $(r,d)$ was a natural choice when thinking about coherent sheaves on $\X$, but it is clear that a simpler normalization exists for Betti tables. Translating results of the previous sections, we will normalize and display them in the standard format
\[
	\begin{tabular}{c|c c }
		  & 0 & 1 \\
		\hline
		0 & $\beta_{0,0}$&$\beta_{1,1}$  \\ 
		1 & $\beta_{0,1}$&$\beta_{1,2}$  \\
		2 & $\beta_{0,2}$&$\beta_{1,3}$  \\ 
		3 & $\beta_{0,3}$&$\beta_{1,4}$  \\ 
		4 & $\beta_{0,4}$&$\beta_{1,5}$  \\
		5 & \vdots & \vdots \\ 
	\end{tabular}
\]
Call an indecomposable $M$ of the first kind if it belongs to the same Auslander-Reiten component as some $\Sigma^{n}k^{st}(m)$, and of the second kind otherwise.
 
\begin{cor}\label{bettifirstkind} The indecomposables of the first kind are uniquely determined by their Betti table. Up to translation, these are all tables of the form 
\begin{align*}
&	\begin{tabular}{c|c c }
		  & 0 & 1 \\
		\hline
		0 & 1&0  \\ 
		1 & r-1&0  \\
		2 & 1&r+1  \\ 
		3 & 0&0  \\ 
	\end{tabular}
\quad
	\begin{tabular}{c|c c }
		  & 0 & 1 \\
		\hline
		0 & r+1&1  \\ 
		1 & 0&r-1 \\
		2 & 0&1  \\ 
		3 & 0&0  \\ 
	\end{tabular}
	\quad \textup{for $r > 0$ odd}, \\
	& \\
&	\begin{tabular}{c|c c }
		  & 0 & 1 \\
		\hline
		0 & 1&0  \\ 
		1 & r&0 \\
		2 & 0&r  \\ 
		3 & 0&1  \\ 
	\end{tabular}
\quad\quad\quad
	\begin{tabular}{c|c c }
		  & 0 & 1 \\
		\hline
		0 & r&0  \\ 
		1 & 1&0 \\
		2 & 0&1  \\ 
		3 & 0&r  \\ 
	\end{tabular}
	\quad\quad \ \  \textup{for $r > 0$ even}.
\end{align*}
These have degree $0$ and rank $r$.
\end{cor}
\begin{proof} Assuming the hypothesis, $\Sigma^{-n}M(-m-2)$ is in the same Auslander-Reiten component as $k^{st}(-2)$ which corresponds to $\mathcal{O}_{\X}$, then apply theorem \ref{ranktwothm} and translate the resulting tables in the above form.
\end{proof}
Most indecomposables are of the second kind.
\begin{cor}\label{bettisecondkind} Up to translation, the Betti tables of indecomposables of the second kind are all tables of type ${\rm I-V}$:
\[
\resizebox{\columnwidth}{!}{%
\begin{tabular}{c|c c }
		I  & 0 & 1 \\
		\hline
		0 & a&0  \\ 
		1 & b&a  \\
		2 & 0&b  \\ 
		3 & 0&0  \\ 
	\end{tabular}
\quad
\quad
	\begin{tabular}{c|c c }
		II  & 0 & 1 \\
		\hline
		0 & a+1&0  \\ 
		1 & b&a  \\
		2 & 0&b+1  \\ 
		3 & 0&0  \\ 
	\end{tabular}
\quad
\quad
	\begin{tabular}{c|c c }
		III  & 0 & 1 \\
		\hline
		0 & a&0  \\ 
		1 & b+1&a+1  \\
		2 & 0&b  \\ 
		3 & 0&0  \\ 
	\end{tabular}
\quad
\quad
	\begin{tabular}{c|c c }
		IV  & 0 & 1 \\
		\hline
		0 & a+2&0  \\ 
		1 & b&a  \\
		2 & 0&b+2  \\ 
		3 & 0&0  \\ 
	\end{tabular}
\quad
\quad
	\begin{tabular}{c|c c }
		V  & 0 & 1 \\
		\hline
		0 & a&0  \\ 
		1 & b+2&a+2  \\
		2 & 0&b  \\ 
		3 & 0&0  \\ 
\end{tabular}%
}
\]
with $a,b \geq 0$, where we have $b \neq 0$ for tables of type ${\rm I}$ and $b-a$ odd for tables of type ${\rm IV-V}$. Here the degree is given by $d = \beta_{0,0} + \beta_{1,2} = 2a,\ 2a+1,\ 2a+2$ and the rank by $r = b-a$.
\end{cor}
\begin{proof} Let $M$ be indecomposable of the second kind. We claim that, up to translation, $\beta(M)$ can be put in the form 
	\[
		\beta(M) = \begin{pmatrix} \beta_{0,0} & \beta_{1,2} \\ \beta_{0,1} & \beta_{1,3} \\ \beta_{0,2} & \beta_{1,4} \\ \beta_{0,3} & \beta_{1,5} \\ \end{pmatrix} = \begin{pmatrix} \alpha & \beta \\ \beta + r & \alpha + r \\ 0 & 0 \\ 0 & 0 \end{pmatrix}
	\]
	for some $\alpha, \beta$ and $r \in \Z$. Running over possibilities in theorem \ref{rankone}, \ref{ranktwothm}, this is already the case in regions $(1), (2)$, where $d, r \geq 0$. For $(rk(M),\ deg(M))$ belonging to region $(3)$, applying $M \mapsto M(2)$ will put $\beta(M)$ in the above form. Note that this sends $(r,d) \mapsto (-r, -d)$, and so the above tables coming from region $(3)$ will have $r < 0$. This will change our fundamental domain to $r > -\frac{d}{2},\ d \geq 0$:
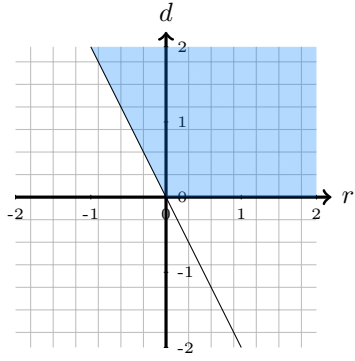
\[
       \begin{tikzpicture}
	    \draw[gray!50, thin, step=0.3] (-2,-2) grid (2,2);
	    \draw[very thick,->] (-2,0) -- (2.2,0) node[right] {$r$};
	    \draw[very thick,->] (0,-2) -- (0,2.2) node[above] {$d$};

	    \foreach \x in {-2,...,2} \draw (\x,0.03) -- (\x,-0.03) node[below] {\tiny\x};
	    \foreach \y in {-2,...,2} \draw (-0.03,\y) -- (0.03,\y) node[right] {\tiny\y};

	    \fill[blue!50!cyan,opacity=0.3] (2,0) -- (0,2) -- (2,2);
	    \fill[blue!50!cyan,opacity=0.3] (0,0) -- (0,2) -- (2,0);
	    \fill[blue!50!cyan,opacity=0.3] (0,0) -- (-1,2) -- (0,2);

	    \draw (-1,2) -- node[below,sloped] {} (1,-2);
       \end{tikzpicture}
\]
Now, the case $\alpha = \beta = 0$ corresponds to $d = 0$, or tables in region $(2)$. The other two regions run over the same pairs $(\alpha, \beta)$, with only difference whether $r \geq 0$ or $r < 0$. Next, set $a = \alpha$, $b = \alpha + r$. Running over the possible $\alpha$ in theorem \ref{rankone}, \ref{ranktwothm} shows that tables must have shapes ${\rm I-V}$. In particular in type ${\rm I}$, note that $b = 0$ implies $r = -\frac{d}{2}$ which falls outside of our domain. Lastly, fixing the type ${\rm I-V}$ of a table, note that $(a,b)$ and $(r,d)$ uniquely determine each other via $r = b-a$ and $d = 2a,\ 2a+1,\ 2a+2$, and so the classification is complete.
\end{proof}

We can also describe indecomposables of the second kind with fixed Betti table. We will say that a family $\mathcal{M}$ of indecomposables is parameterized by $\X$ (resp. $\mathcal{U} \subseteq \X$) of level $n$ if there is a fully faithful functor 
\[
	\Phi_{\mathcal{M}}: \mathcal{C}_{\infty} \hookrightarrow \underline{\rm MCM}(gr R)
\]
which sends the skyscraper sheaves of degree $2n$ to $\mathcal{M}$ (resp. skyscraper sheaves supported over $\mathcal{U}$). As corollaries of theorem \ref{rankone}, \ref{ranktwothm}, we have

\begin{cor} There are four indecomposables for each table of type ${\rm II-III}$, and a unique indecomposable for each table of type ${\rm IV-V}$. Type ${\rm I}$ breaks down as follows for fixed $(a,b)$:
	\begin{enumerate}[i.]
		\item For $r = b-a$ odd, there are six indecomposables per table.
		\item For $d = 2a \neq 0$ and $r = b-a$ even, indecomposables are parameterized by $\X$ of level $n = \frac{gcd(r,d)}{2}$.
		\item For $d = 2a = 0$ and $r = b$ even, indecomposables are parameterized by $\X \setminus \{\infty\}$ of level $n = \frac{r}{2}$. 
	\end{enumerate}
	
\end{cor}

\subsection*{Ulrich modules} Let $M$ be a (graded) MCM $R$-module. Let $\mu = \mu(M)$ denote the minimal number of generators of a module $M$, and $e = e(M)$ its multiplicity. The latter can be calculated from the Hilbert series of $M$, given by
\[
	H_M(t) = \frac{P_M(t)}{1-t}
\]
for some Laurent polynomial $P_M(t)$ with $P_M(1) \neq 0$. We then have $e(M) = P_M(1)$. Note that $e\left( M(1) \right) = e(M)$. There is a general bound $\mu(M) \leq e(M)$, and Ulrich modules are defined as MCM modules meeting this bound. We can calculate the Hilbert series $H_M(t)$ from the Betti table $\beta(M)$ as
\begin{align*}
	H_M(t) 	&= \frac{H_R(t)}{1-t^4}\left(\sum_{j \in \Z} \beta_{0, j}t^j - \sum_{j \in \Z} \beta_{1,j} t^j\right)
\end{align*}
and so we will deduce a classification of Ulrich modules. We first note that $R$ possesses some Ulrich modules: the module $L_i = R/l_i$ is cyclic with Hilbert series $\frac{1}{1-t}$, and so $\mu = e = 1$. 

\begin{thm} Let $M$ be an indecomposable graded Ulrich module. Then up to degree shift we have $M \cong L_i$ for some $i = 1,2,3,4$.
\end{thm}
\begin{proof} First assume that $M$ is of second kind, so that up to grade shift $M$ has Betti table
	\[
	\begin{tabular}{c|c c }
		  & 0 & 1 \\
		\hline
		0 & $\beta_{0,0}$ & $\beta_{1,1}$ \\ 
		1 & $\beta_{0,1}$ & $\beta_{1,2}$  \\
		2 & $\beta_{0,2}$ & $\beta_{1,3}$  \\ 
		3 & $\beta_{0,3}$ & $\beta_{1,4}$  \\ 
	\end{tabular}
	\quad
	=
	\quad
	\begin{tabular}{c|c c }
		  & 0 & 1 \\
		\hline
		0 & $\alpha$ & 0  \\ 
		1 & $\beta + r$ & $\beta$  \\
		2 & 0& $\alpha + r$  \\ 
		3 & 0&0  \\ 
	\end{tabular}
	\]
	for some $\alpha, \beta \geq 0$ and $r \in \Z$. Then
	\begin{align*}
		H_M(t)  &= \frac{1}{1-t^4}H_R(t) \left(\alpha t^0 + (\beta + r)t^1 - \beta t^2 - (\alpha + r)t^3 \right)\\
		&= \frac{1}{1-t^4}\frac{1-t^4}{(1-t)^2}\left(\alpha t^0 + (\beta + r)t^1 - \beta t^2 - (\alpha + r)t^3 \right)\\
		&= \frac{1}{(1-t)^2} \left(\alpha(1-t^2) + \beta (t-t^3) + (\alpha - \beta)(t^2 - t^3) + r(t-t^3) \right)\\
		&= \frac{1}{(1-t)^2}(1-t)\left(\alpha(1+t) +  \beta t (1+t) + (\alpha - \beta) t^2 +rt(1+t) \right)
	\end{align*}
	and so $e = 3\alpha + \beta + 2r \geq \alpha + \beta + r = \mu$, with equality when $\alpha = r = 0$. The only possibility is given by the type {\rm III} table 
	\[
	\begin{tabular}{c|c c }
		  & 0 & 1 \\
		\hline
		0 & 0&0  \\ 
		1 & 1&1  \\
		2 & 0&0  \\ 
		3 & 0&0  \\ 
	\end{tabular}
	\]
	There are four indecomposables with this table, given by $L_i(-1),\ i = 1,2,3,4$. One easily verifies that $e > \mu$ for indecomposables of the first kind, and so the result holds.
\end{proof}

\section{Matrix factorizations corresponding to simple torsion sheaves}
We will produce the indecomposable matrix factorizations corresponding to the simple objects in $\mathcal{C}_\infty$. For simplicity we assume $char\ k \neq 2$ in this section. The pair $\left( k^{st}(-1)[-1], k^{st}(-2)\right)$ corresponds to $\left(\mathcal{O}(-\vec{c}), \mathcal{O}\right)$, with $2$-dimensional morphism space. Taking a basis $\phi_0, \phi_\infty$, for any $p = [p_0 : p_1] \in \mathbb{P}^1(k)$ we define the MCM module $M_p$ as the cone of $\phi_p = p_1 \phi_0 + p_0 \phi_\infty$
\[
	k^{st}(-1)[1] \xrightarrow{\phi_p} k^{st}(-2) \to M_p \to k^{st}(-1)[2]
\]
whose distinguished triangle is sent to
\[
	\mathcal{O}(-\vec{c}) \xrightarrow{s_p} \mathcal{O} \to {\rm coker}(s_p) \to \mathcal{O}(-\vec{c})[1].
\]
for the corresponding cosection $s_p$, with cokernel an ordinary skyscraper sheaf. We can produce the associated matrix factorizations. We have $f_\lambda = xy(x-y)(x- \lambda y) = x^3y - (1+\lambda)x^2y^2 + \lambda xy^3$. Write $f_\lambda = x f_x + y f_y$ for $f_x = \frac{1}{4} \frac{\partial f_\lambda}{\partial x}$ and $f_y = \frac{1}{4}\frac{\partial f_\lambda}{\partial y}$, and note that $x | f_y$ and $y | f_x$.
\begin{prop}\label{stabresidue} We have the following explicit presentations.
	\begin{enumerate}
		\item $k^{st}$ corresponds to the matrix factorization
			\[
				\xymatrix@C2pc{S \oplus S(2) & \ar[l]_-{A} S(-1) \oplus S(-1) & \ar[l]_-{B} S(-4) \oplus S(-2)}
			\]
			with
			\begin{align*}
				A = \begin{pmatrix} x & y \\ -f_y & f_x \end{pmatrix} && B = \begin{pmatrix} f_x & -y \\ f_y & x \end{pmatrix}.
			\end{align*}
		\item The morphisms $\phi_0, \phi_\infty: k^{st}(-1)[-1] \to k^{st}(-2)$ can be realized as 
			\[
				\xymatrix@C2pc{S(-2) \oplus S(-2) \ar@<.5ex>[d]^{\varphi_\infty} \ar@<-.5ex>[d]_{\varphi_0}    & \ar[l]_-{-B} S(-5) \oplus S(-3) \ar@<.5ex>[d]^{\psi_\infty} \ar@<-.5ex>[d]_{\psi_0}  & \ar[l]_-{-A} S(-6) \oplus S(-6) \ar@<.5ex>[d]^{\varphi_\infty} \ar@<-.5ex>[d]_{\varphi_0}  \\
				S(-2) \oplus S & \ar[l]_-{A} S(-3) \oplus S(-3) & \ar[l]_-{B} S(-6) \oplus S(-4)}	
			\]
			with matrices given by
			\begin{align*}
				&\varphi_0 = \begin{pmatrix} 0 & 1 \\ 0 & -\frac{f_y}{x}  \end{pmatrix} && \psi_0 = \begin{pmatrix} -\frac{f_y}{x} & -1 \\ 0 & 0\end{pmatrix}\\
				&\varphi_\infty = \begin{pmatrix} 1 & 0 \\ \frac{f_x}{y} & 0  \end{pmatrix} && \psi_\infty = \begin{pmatrix} 0 & 0\\ -\frac{f_x}{y} & 1\end{pmatrix}.
			\end{align*}
	\end{enumerate}
\end{prop}
\begin{proof} Part (1) follows from the Tate resolution \cite{Eis}. For part (2), note that the MCM approximation $k^{st}(-2) \to k(-2)$ corresponds to the natural projection and induces natural isomorphisms on Tate cohomology
\[
	\underline{{\rm Ext}}^0_{gr R}\left(k^{st}(-1)[-1], k^{st}(-2)\right) \xrightarrow{\cong} \underline{{\rm Ext}}^0_{gr R}\left(k^{st}(-1)[-1], k(-2)\right).
\]
The morphisms $\phi_0, \phi_\infty$ descend to the natural basis on the latter. \end{proof}
Now let $\phi_p = p_1 \phi_0 + p_0 \phi_\infty$. Taking $M_p = Cone(\phi_p)$ yields a $4 \times 4$ matrix factorization
\[
	\xymatrix@R0.01pc@C5pc{S(-1) & S(-2) & S(-5)\\
	S(1) & \ar@<2ex>[l]_{\begin{pmatrix} A & 0 \\ -\varphi_p & A \end{pmatrix}} S(-2) & \ar@<2ex>[l]_{\begin{pmatrix} B & 0 \\ -\psi_p & B \end{pmatrix}} S(-3)\\
		S(-2) & S(-3) & S(-6)\\
		S & S(-3) & S(-4)
	}
\]
Comparing with the classification of Betti tables, this matrix factorization must be stably equivalent to a $2 \times 2$ matrix factorization. Direct calculations show the following:
\begin{prop}\label{ordinaryskyscraper} The module $M_p$ is given by the reduced matrix factorization
			\[
				\xymatrix@C2pc{S(-1) \oplus S & \ar[l]_-{A_p} S(-2) \oplus S(-3) & \ar[l]_-{B_p} S(-5) \oplus S(-4)}
			\]
			where $(A_p, B_p)$ for $p_1 \neq 0$ are given by
			\begin{align*}
				A_p = \begin{pmatrix} x - \frac{p_0}{p_1} y & \frac{1}{p_1}y^2 \\ -p_0 \frac{f_\lambda}{xy} & \frac{f_\lambda}{x} \end{pmatrix} && B_p = \begin{pmatrix} \frac{f_\lambda}{x} & -\frac{1}{p_1}y^2 \\ p_0 \frac{f_\lambda}{xy} & x - \frac{p_0}{p_1} y \end{pmatrix} 
			\end{align*}
			and for $p_1 = 0$, $p_0 \neq 0$ by
			\begin{align*}
				A_p = \begin{pmatrix} y & x^2 \\ 0 & \frac{f_\lambda}{y} \end{pmatrix} && B_p = \begin{pmatrix} \frac{f_\lambda}{y} & -x^2 \\ 0 & y \end{pmatrix}
			\end{align*}
\end{prop}

Recall that the modules $L_i(-3)[1] = R/(\frac{f_\lambda}{l_i})$ correspond to the simple sheaves $S_{i, 0}$, where $L_i = R/l_i$  . From the above presentation, one verifies
\begin{lem} For each $p_i = V(l_i) \in \mathbb{P}^1(k)$, there are short exact sequences of MCM modules
	\[
		0 \to	R/l_i(-1) \to M_{p_i} \to R/(\frac{f_\lambda}{l_i}) \to 0.
	\]
\end{lem}
Hence the $M_{p_i}$ corresponds to $S_{p_i}$. Summarising, we have shown
\begin{prop} The indecomposable MCM modules 
	\[
		\{ M_p \}_{p \neq 0, 1, \infty, \lambda} \cup \{R/(\frac{f_\lambda}{l_i}), R/l_i(-1) \}_{i=1,2,3,4}
	\]
	correspond under the equivalence of corollary \ref{wpl} to the simple torsion sheaves in $\mathcal{C}_\infty$. 
\end{prop} 
The remaining indecomposables are constructed from these by taking extensions and applying tubular mutations.  
\bibliographystyle{plain}
\bibliography{elliptic}

\end{document}